\documentclass[12pt]{article}
\usepackage{amsmath,amsfonts,amsthm,amssymb}
\usepackage{graphicx}
\usepackage{fancyhdr}
\usepackage{enumerate}
\usepackage[justification=centering]{caption}
\usepackage[labelformat=simple]{subcaption}

\usepackage{mathrsfs}

\providecommand{\bA}{\boldsymbol{A}}

\newtheorem{theorem}{Theorem}[section]
\newtheorem{proposition}[theorem]{Proposition}
\newtheorem{lemma}[theorem]{Lemma}
\newtheorem{corollary}[theorem]{Corollary}
\theoremstyle{definition}

\newtheorem{example}[theorem]{Example}

\makeatletter
\@addtoreset{equation}{section}

\makeatother  

\providecommand{\floor}[1]{\left\lfloor#1\right\rfloor}

\providecommand{\bA}{\boldsymbol{A}}

\providecommand{\be}{\begin{equation}}
\providecommand{\ee}{\end{equation}}
\def\ba#1\ea{\begin{align}#1\end{align}}
\def\bas#1\eas{\begin{align*}#1\end{align*}}

\begin{document}
\title{
A spectral version of the Moore problem for bipartite regular graphs}
\author{Sebastian M. Cioab\u{a}$^{*}$, Jack H. Koolen$^{\dagger}$,  and 
Hiroshi Nozaki$^\ddagger$ \\
\small $^*$Department of Mathematical Sciences,\vspace{-3pt}\\
\small University of Delaware, Newark DE 19716-2553, USA\vspace{4pt}\\
\small $^\dagger$ School of Mathematical Sciences,\vspace{-3pt}\\
\small  University of Science and Technology of China,\vspace{-3pt}\\
\small Wen-Tsun Wu Key Laboratory of the Chinese Academy of Sciences, Hefei, Anhui, China \vspace{4pt}\\
\small $^\ddagger$  Department of Mathematics Education,\vspace{-3pt}\\
\small   Aichi University of Education, 1 Hirosawa, Igaya-cho, Kariya, Aichi 448-8542, Japan \vspace{4pt}\\
\small {\tt cioaba@udel.edu, koolen@ustc.edu.cn, hnozaki@auecc.aichi-edu.ac.jp}}
\date{\today}
\maketitle

\footnotetext[1]{SMC is supported by the NSF grants DMS-1600768 and CCF-1815992.}
\footnotetext[2]{JHK is partially supported by the National Natural Science Foundation of China (Nos.\ 11471009 and 11671376). He also acknowledges the financial support of the Chinese Academy of Sciences under its `100 talent' program.}
\footnotetext[3]{HN is partially supported by JSPS Grants-in-Aid for Scientific Research Nos.\ 16K17569 and 17K0515501. 
}
\begin{abstract}
Let $b(k,\theta)$ be the maximum order of a connected bipartite $k$-regular graph whose second largest eigenvalue is at most $\theta$. In this paper, we obtain a general upper bound for $b(k,\theta)$ for any $0\leq \theta< 2\sqrt{k-1}$. Our bound gives the exact value of $b(k,\theta)$ whenever there exists a bipartite distance-regular graph of degree $k$, second largest eigenvalue $\theta$, diameter $d$ and girth $g$ such that $g\geq 2d-2$. For certain values of $d$, there are infinitely many such graphs of various valencies $k$. However, for $d=11$ or $d\geq 15$, we prove that there are no bipartite distance-regular graphs with $g\geq 2d-2$. 
\end{abstract}

\section{Introduction}

Let $\Gamma=(V,E)$ be a connected $k$-regular simple graph with $n$ vertices. For $1\leq i\leq n$, let $\lambda_i(\Gamma)$ denote the $i$-th largest eigenvalue of the adjacency matrix of $\Gamma$. The eigenvalues have close relationships with other graph invariants. The smallest eigenvalue $\lambda_n(\Gamma)$ is related to the diameter, the chromatic number and the independence number (see \cite{Ch89} or \cite[Chapter 4]{BH12} for example). The second eigenvalue $\lambda_2(\Gamma)$ plays a fundamental role in the study of expanders \cite{A86, AM85, BH12, HLW06}. Let $v(k,\theta)$ denote the maximum order of a connected $k$-regular graph $\Gamma$ with $\lambda_2(\Gamma) \leq \theta$. For $\theta < 2\sqrt{k-1}$, from work of Alon and Boppana, and Serre, we know that the value $v(k,\theta)$ is finite (see \cite{CKNV16} and the references therein). In \cite{CKNV16}, we obtained the following upper bound for $v(k,\theta)$. Let $\boldsymbol{T}(k,t,c)$ be the $t \times t$ tridiagonal matrix  with lower diagonal $(1,1,\ldots,1,c)$, upper diagonal $(k,k-1,\ldots, k-1)$, and with constant row sum $k$.  If $\theta$ is the second largest eigenvalue of $\boldsymbol{T}(k,t,c)$, then
\begin{equation}\label{vktau}
v(k,\theta) \leq 1+\sum_{i=0}^{t-3} k(k-1)^i+ \frac{k(k-1)^{t-2}}{c}. 
\end{equation}
Equality holds in \eqref{vktau} if and only if there is a distance-regular graph of valency $k$ with second largest eigenvalue $\theta$, girth $g$ and diameter $d$ satisfying $g\geq 2d$. For $d>6$, there are no such graphs \cite{DG81}. However, for smaller values of $d$, there are infinitely many values of $k$ and $\theta$ where the above inequality gives the exact value of $v(k,\theta)$. 

In this paper, we improve the above results from \cite{CKNV16} for bipartite regular graphs. Let $b(k,\theta)$ denote the maximum order of a connected bipartite $k$-regular graph $\Gamma$ with $\lambda_2(\Gamma) \leq \theta$.  Bipartite regular graphs $\Gamma$ with $\lambda_2(\Gamma) \leq \theta$  have been classified for $\theta=\sqrt{2}$ \cite{TY00}, $\theta=\sqrt{3}$ \cite{KS13}, and $\theta=2$~\cite{KS14}. We obtain a general upper bound for $b(k,\theta)$ for any $0\leq \theta< 2\sqrt{k-1}$. Our bound gives the exact value of $b(k,\theta)$ whenever there exists a bipartite distance-regular graph of degree $k$ with second largest eigenvalue $\theta$, diameter $d$ and girth $g$ such that $g\geq 2d-2$. For certain values of $d$, there are infinitely many such graphs of various valencies $k$. When $d=11$ or $d\geq 15$, we prove the non-existence of bipartite distance-regular graphs with $g\geq 2d-2$. Our results generalize previous work of H\o holdt and Justesen \cite{HJ14} obtained in their study of graph codes and imply some results of Li and Sol\'{e} \cite{CLS96} relating the second largest eigenvalue of a bipartite regular graph to its girth. The degree-diameter or Moore problem for graphs \cite{MS} is about determining the largest graphs of given maximum degree and diameter. Given the connections between the diameter and the second largest eigenvalue of bipartite regular graphs (see \cite{Ch89} for example), our Theorem \ref{bound} can be interpreted as a spectral version of the Moore problem for bipartite regular graphs. 

In Section~\ref{sec:pre}, we describe some sequences of orthogonal polynomials and develop the preliminary results and notation that will be used in the paper. In Section~\ref{sec:lp}, we improve the linear programming bound from \cite{N} for the class of bipartite regular graphs.  In Section~\ref{sec:newbound}, we obtain the following upper bound for $b(k,\theta)$. Let $\boldsymbol{B}(k,t,c)$ be the $t \times t$ tridiagonal matrix  with lower diagonal $(1,\ldots,1,c,k)$, upper diagonal $(k,k-1,\ldots, k-1,k-c)$, and constant row sum $k$. 
If $\theta$ is the second largest eigenvalue of $\boldsymbol{B}(k,t,c)$, then
\begin{equation}\label{bktau}
b(k,\theta) \leq 2\left(\sum_{i=0}^{t-4} (k-1)^{i}+ \frac{(k-1)^{t-3}}{c}+\frac{(k-1)^{t-2}}{c}\right).  
\end{equation}
We show that equality happens in \eqref{bktau} when there is a bipartite distance-regular graph of degree $k$, second largest eigenvalue $\theta$ having $g\geq 2d-2$. Inequality \eqref{bktau} generalizes some results of H\o holdt and Justesen \cite{HJ14} (see Corollaries~\ref{coro:0} and \ref{coro:1}), and of Li and Sol\'{e} \cite{CLS96} (see Corollary~\ref{coro:Alon}). 
At the end of Section~\ref{sec:newbound}, we prove that the bound \eqref{bktau} is better than \eqref{vktau} for any $k$ and $\theta$. In Section~\ref{sec:bound_d}, we prove the non-existence of bipartite distance-regular graphs with $g\geq 2d-2$ for $d=11$ and $d\geq 15$. We conclude the paper with some remarks in Section~\ref{sec:conclusions}.

\section{Preliminaries} \label{sec:pre}
In this section, we describe some useful polynomials that will be used to prove our main result. For any integer $k\geq 2$, let $(F_i^{(k)})_{i\geq 0}$ be a sequence of orthogonal polynomials defined by the three-term recurrence relation:  
\begin{equation*}
F_0^{(k)}(x)=1, \qquad F_1^{(k)}(x)=x, \qquad  F_2^{(k)}(x)=x^2 - k,   
\end{equation*}
and
\begin{equation} \label{eq:3-term}
F_i^{(k)}(x)=x F_{i-1}^{(k)}(x)- (k-1) F_{i-2}^{(k)}(x)
\end{equation}
for $i\geq 3$. 
The notation $F_i^{(k)}$ is abbreviated to $F_i$ for the rest of the paper. Let $q=\sqrt{k-1}$. The polynomials $(F_i)_{i\geq 0}$ form a sequence of orthogonal polynomials 
with respect to the positive weight
\[
w(x)=\frac{\sqrt{4q^2-x^2}}{k^2-x^2}
\]
on the interval $[-2q, 2q]$ 
(see \cite[Section 4]{HOb}). 
The polynomials $F_i(q y)/ q^i$ in $y$ are called Geronimus
polynomials \cite{G30,H15}. 
It follows from \eqref{eq:3-term} that 
\begin{equation} \label{eq:3-term2}
F_i(x)=(x^2-2k+2) F_{i-2}(x)- (k-1)^2 F_{i-4}(x)
\end{equation}
for $i\geq 5$. Note that for any $i\geq 0$, $F_{2i}(x)$ and $F_{2i+1}(x)$ are 
even and odd functions of $x$, respectively.

For $i\geq 0$, let $\mathscr{F}_{0,i}(x)=F_{2i}(\sqrt{x})$ and $\mathscr{F}_{1,i}(x)=F_{2i+1}(\sqrt{x})/\sqrt{x}$. 
It follows that $x^{\epsilon} \mathscr{F}_{\epsilon,i}(x^2)=F_{2i+\epsilon}(x)$ for $\epsilon\in \{0,1\}$. 
By \eqref{eq:3-term2}, the polynomials $\mathscr{F}_{0,i}(x)$ and $\mathscr{F}_{1,i}(x)$  satisfy the following properties:
\[
\mathscr{F}_{0,0}(x)=1, \qquad 
\mathscr{F}_{0,1}(x)=x-k, \qquad
\mathscr{F}_{0,2}(x)=x^2-(3k-2)x+k(k-1), 
\]
\[
\mathscr{F}_{1,0}(x)=1, \qquad 
\mathscr{F}_{1,1}(x)=x-(2k-1),
\]
and 
\begin{equation} \label{eq:new_3-term}
\mathscr{F}_{\epsilon, i}(x)=(x-2k+2) \mathscr{F}_{\epsilon,i-1}(x)-(k-1)^2 \mathscr{F}_{\epsilon, i-2}(x)
\end{equation}
for any $i\geq 3$ if $\epsilon=0$, and $i\geq 2$ if $\epsilon=1$.  
Note that $k^{\epsilon} \mathscr{F}_{\epsilon,i}(k^2)=F_{2i+\epsilon}(k)=k(k-1)^{2i-1+\epsilon}=(k-1)^{2i-1+\epsilon}+(k-1)^{2i+\epsilon}$ for $2i+\epsilon \ne 0$. 
For $\epsilon\in \{0,1\}$, the polynomials $(\mathscr{F}_{\epsilon,i})_{i\geq 0}$  form a sequence of orthogonal polynomials with respect to the positive weight 
\[
w_\epsilon(x)=\frac{x^{\epsilon-1/2} \sqrt{4q^2-x}}{k^2-x}
\]
on the interval $[0, 4q^2]$.  

For $i\geq 0$, let $G_i(x)=\sum_{j=0}^{\floor{i/2}}F_{i-2j}(x)$. A simple calculation implies that
\begin{equation} \label{eq:G_i}
{G}_{i}(x)=
\frac{{F}_{i+2}(x)-(k-1)^2 {F}_{i}(x) }{x^2-k^2}
\end{equation}
for $i\geq 1$.
From Lemmas 3.3 and 3.5 in \cite{CK07}, the polynomials $(G_i)_{i\geq 0}$  form a sequence of 
orthogonal polynomials with respect to 
the positive weight $(k^2-x^2)w(x)=\sqrt{4q^2-x^2}$ on the interval $[-2q,2q]$. From \eqref{eq:3-term}, we deduce that
\[
G_i(x)=xG_{i-1}(x)-(k-1)G_{i-2}(x)
\]
for $i\geq 2$. 

Let $\mathscr{G}_{\epsilon, i}(x)$ denote the 
polynomial 
\begin{equation} \label{eq:def_scrG}
\mathscr{G}_{\epsilon, i}(x)=\sum_{j=0}^i
\mathscr{F}_{\epsilon, j}(x). 
\end{equation} 
It follows that 
$x^{\epsilon} \mathscr{G}_{\epsilon, j}(x^2)=G_{2j+\epsilon}(x)$. Using \eqref{eq:G_i}, the polynomial $\mathscr{G}_{\epsilon, i}(x)$ can be expressed as
\begin{equation} \label{eq:pro_scrG}
\mathscr{G}_{\epsilon, i}(x)=
\frac{\mathscr{F}_{\epsilon, i+1}(x)-(k-1)^2\mathscr{F}_{\epsilon, i}(x) }{x-k^2}
\end{equation}
for any $i\geq 2$ if $\epsilon=0$, and
$i\geq 1$ if $\epsilon=1$. 
From Lemmas 3.3 and 3.5 in \cite{CK07}, 
for $\epsilon\in \{0,1\}$, the polynomials $(\mathscr{G}_{\epsilon,i})_{i\geq 0}$ form  a sequence of orthogonal polynomials with respect to 
the positive weight $(k^2-x)w_{\epsilon}(x)=x^{\epsilon-1/2}\sqrt{4q^2-x}$ on the interval $[0,4q^2]$. 

\begin{lemma} \label{lem:coe}
Let $p_l(i,j)$ be the coefficients in 
$x^{\epsilon} \mathscr{F}_{\epsilon,i}(x)\mathscr{F}_{\epsilon,j}(x)=\sum_{l=0}^{i+j+\epsilon}p_l(i,j)
\mathscr{F}_{0,l}(x)$. Then
we have 
$p_0(i,j)=
k^{\epsilon}\mathscr{F}_{\epsilon,i}(k^2)\delta_{i,j}$, and 
$p_l(i,j)\geq 0$ for any $l,i,j$. Moreover 
$p_l(i,j)>0$ if and only if $|i-j|\leq l \leq i+j+\epsilon$.  
\end{lemma}
\begin{proof}
We have that
\[
F_{2i+\epsilon}(x)F_{2j+\epsilon}(x)=
x^{2\epsilon}\mathscr{F}_{\epsilon,i}(x^2)\mathscr{F}_{\epsilon,j}(x^2)
=
\sum_{l=0}^{i+j+\epsilon}p_l(i,j)
\mathscr{F}_{0,l}(x^2)
=
\sum_{l=0}^{i+j+\epsilon}p_l(i,j)F_{2l}(x).
\]
By Theorem 3 in \cite{N}, we obtain that 
$p_0(i,j)=F_{2i+\epsilon}(k) \delta_{i,j}=
k^{\epsilon}\mathscr{F}_{\epsilon,i}(k^2)\delta_{i,j}$, and 
$p_l(i,j)\geq 0$ for any $l,i,j$. Moreover 
$p_l(i,j)>0$ if and only if $|i-j|\leq l \leq i+j+\epsilon$.  
\end{proof}
Let $\Gamma$ be a connected regular bipartite graph. 
The adjacency matrix $\bA$ of $\Gamma$ can be expressed by 
\[
\boldsymbol{A}=
\begin{pmatrix}
\boldsymbol{O} & \boldsymbol{N}\\
\boldsymbol{N}^{\top} & \boldsymbol{O} 
\end{pmatrix}, 
\]
where $\boldsymbol{N}^{\top}$ is the transpose matrix of $\boldsymbol{N}$. 
The matrix $\boldsymbol{N}$ is called the {\it biadjacency matrix} of $\Gamma$. 
It is not hard to see that 
\begin{equation} \label{eq:AB}
F_{2i}(\boldsymbol{A})= 
\begin{pmatrix}
\mathscr{F}_{0,i}(\boldsymbol{N}\boldsymbol{N}^{\top}) & \boldsymbol{O} \\
\boldsymbol{O} & \mathscr{F}_{0,i}(\boldsymbol{N}^{\top} \boldsymbol{N})
\end{pmatrix}.  
\end{equation}
Since each entry of $F_{2i}(\boldsymbol{A})$ is non-negative \cite{S66}, 
each entry of $\mathscr{F}_{0,i}(\boldsymbol{N}\boldsymbol{N}^{\top})$ is also non-negative. 

\section{Linear programming bound for bipartite regular graphs} 
\label{sec:lp}

In this section, we give a linear programming bound for bipartite regular graphs. For general regular graphs, a linear programming bound was obtained by Nozaki \cite{N}. 
\begin{theorem}
\label{thm:lp_bound}
Let $\Gamma$ be a connected bipartite $k$-regular graph with $v$ vertices. 
Let $\{\pm \tau_0, \ldots, \pm \tau_d \}$ be the set of distinct eigenvalues of $\Gamma$, 
where $\tau_0=k$.  
If there exists a polynomial $f(x)=\sum_{i= 0}^t  f_i \mathscr{F}_{0,i}(x)$ such that 
$f(k^2)>0$, $f(\tau_i^2) \leq 0$ for each $i\in  \{1,\ldots, d \}$, $f_0>0$, and $f_j \geq 0$ for each $j \in \{1,\ldots, t \}$, then 
\begin{equation} 
\label{eq:lp} 
v \leq  \frac{2 f(k^2)}{f_0}. 
\end{equation}
Equality holds if and only if for each $i \in \{1,\ldots, d\}$, $f(\tau_i^2)=0$ and for each $j\in \{1,\dots,t\}$, ${\rm tr}(f_j\mathscr{F}_{0,j}(\boldsymbol{N}\boldsymbol{N}^{\top}))=0$, 
and ${\rm tr} (f_j\mathscr{F}_{0,j}(\boldsymbol{N}^{\top} \boldsymbol{N}))=0$, 
where $\boldsymbol{N}$ is the biadjacency matrix of $\Gamma$.
If equality holds and $f_j>0$ for each $j \in \{1,\ldots, t\}$, then the girth of $\Gamma$ is at least  
$2t+2$. 
\end{theorem}
\begin{proof}
From the spectral decomposition $\boldsymbol{N}\boldsymbol{N}^{\top}= 
\sum_{i=0}^d \tau_i^2 \boldsymbol{E}_i$, we deduce that  
\begin{equation} \label{eq:pr1}
f(k^2) \boldsymbol{E}_0+ \sum_{i=1}^d f(\tau_i^2) \boldsymbol{E}_i
=f(\boldsymbol{N}\boldsymbol{N}^{\top})=\sum_{i= 0}^t  f_i \mathscr{F}_{0,i}(\boldsymbol{N}\boldsymbol{N}^{\top})=
f_0 \boldsymbol{I} + \sum_{i= 1}^t  f_i \mathscr{F}_{0,i}(\boldsymbol{N}\boldsymbol{N}^{\top}), 
\end{equation}
where  $\boldsymbol{I}$ is the identity matrix, $\boldsymbol{E}_0=(2/v)\boldsymbol{J}$, and $\boldsymbol{J}$ is the all-ones matrix. 
Taking traces in both sides of \eqref{eq:pr1}, we get that
\begin{multline*}
 f(k^2)={\rm tr}( f(k^2) \boldsymbol{E}_0 )
\geq {\rm tr}\left( f(k^2) \boldsymbol{E}_0+ \sum_{i=1}^d f(\tau_i^2) \boldsymbol{E}_i\right) \\
= 
{\rm tr}\left(  f_0 \boldsymbol{I} + \sum_{j= 1}^t  f_j\mathscr{F}_{0,j}(\boldsymbol{N}\boldsymbol{N}^{\top}) \right) 
\geq {\rm tr}(  f_0 \boldsymbol{I})
=\frac{v f_0}{2}. 
\end{multline*}
Therefore, $v\leq 2 f(k^2)/f_0$.  
By using $\mathscr{F}_{0,j}(\boldsymbol{N}^{\top} \boldsymbol{N})$, 
we can obtain the same bound as \eqref{eq:lp}. 

If equality holds in \eqref{eq:lp}, then for each $i \in \{1,\ldots, d\}$, $f(\tau_i^2)=0$ and for each $j\in \{1,\dots,t\}$, ${\rm tr}(f_j\mathscr{F}_{0,j}(\boldsymbol{N}\boldsymbol{N}^{\top}))=0$ and ${\rm tr} (f_j\mathscr{F}_{0,j}(\boldsymbol{N}^{\top} \boldsymbol{N}))=0$. 
 For the adjacency matrix $\boldsymbol{A}$,  the $(u,v)$-entry of $F_{j}(\boldsymbol{A})$ is the number of non-backtracking walks of length $j$ from $u$ to $v$ \cite{S66}. 
Since \eqref{eq:AB} and $f_j>0$ for each $j \in \{1,\ldots, t\}$, there is no non-backtracking walk of length $2j$ from $u$ to $v$ for each $j \in \{1,\ldots, t\}$. Since $\Gamma$ is bipartite, the girth of $\Gamma$ is at least  $2t+2$. 
\end{proof}


\section{Upper bound for bipartite graphs with given second eigenvalue}
\label{sec:newbound}

In this section, we obtain an upper bound on $b(k,\theta)$ using the {\em bipartite} linear programming bound given by Theorem \ref{thm:lp_bound}. Let $c>0$ be a real number and $t\geq 4$ be an integer. Let $\boldsymbol{B}(k,t,c)$ be the $t \times t$ tridiagonal 
matrix 
with lower diagonal $(1,\ldots,1,c,k)$, upper diagonal $(k,k-1,\ldots, k-1,k-c)$, and constant row sum $k$. 
Let
\[
\boldsymbol{B}(k,3,1)=\begin{pmatrix}
0&k&0 \\
1&0&k-1 \\
0&k& 0
\end{pmatrix}.
\]

\begin{theorem} \label{bound}
If $\theta$ is the second largest eigenvalue of $\boldsymbol{B}(k,t,c)$, then
\[
b(k,\theta) \leq M(k,t,c)=2\left(\sum_{i=0}^{t-4} (k-1)^{i}+ \frac{(k-1)^{t-3}}{c}+\frac{(k-1)^{t-2}}{c}\right). 
\]  
Equality holds if and only if there exists a bipartite distance-regular graph whose quotient matrix with respect to the distance-partition from a vertex is $\boldsymbol{B}(k,t,c)$ for $1\leq c<k$ or $\boldsymbol{B}(k,t-1,1)$ for $c=k$.   
\end{theorem} 

\begin{proof}
We first calculate 
the characteristic polynomial of $\boldsymbol{B}(k,t,c)$. 
The polynomials $F_i$, $G_i$, $\mathscr{F}_i$, and $\mathscr{G}_i$  are defined in Section~\ref{sec:pre}. 
Note that $F_i(x)$ is the characteristic polynomial of the principal $i \times i$ matrix formed by the first 
$i$ rows and $i$ columns of $\boldsymbol{B}(k,t,1)$ for $t>i+1$. 
By this fact and equations \eqref{eq:3-term2} and \eqref{eq:G_i}, we can compute 
\begin{align*}
|x \boldsymbol{I}-\boldsymbol{B}(k,t,c)|
&={\small \begin{vmatrix}
x   & -k       &            &           &          & \\
-1 &  x        & -(k-1)  &            &         & \\
    & \ddots  & \ddots & \ddots &         &  \\
    &             &-1       & x         & -(k-1)&    \\
    &             &          & -c       & x        & -(k-c) \\
    &             &           &          & -k      & x     
\end{vmatrix}}\\
&=k {\small \begin{vmatrix}
x   & -k       &            &                   & \\
-1 &  x        & -(k-1)  &                  & \\
    & \ddots  & \ddots & \ddots          &  \\
    &             &-1       & x          &   0 \\
    &             &          & -c          & -(k-c) \\    
\end{vmatrix}}+x{\small \begin{vmatrix}
x   & -k       &            &           &           \\
-1 &  x        & -(k-1)  &            &          \\
    & \ddots  & \ddots & \ddots &          \\
    &             &-1       & x         & -(k-1)    \\
    &             &          & -c       & x            
\end{vmatrix}}\\
&=-k(k-c)F_{t-2}(x)+x(x F_{t-2}(x)-c(k-1)F_{t-3}(x))\\
&= c\big(F_{t-2}(x)-(k-1)^2 F_{t-4}(x)\big)+(x^2-k^2)F_{t-2}(x)\\
&= c\big((x^2-k^2)F_{t-4}(x)+F_{t-4}(x)-(k-1)^2 F_{t-6}(x)\big)+(x^2-k^2)F_{t-2}(x)\\
&= (x^2-k^2)\left(c\sum_{i=0}^{ \lfloor (t-4)/2 \rfloor}F_{t-4-2i}(x)+F_{t-2}(x)\right).
\end{align*}
Note that
\[
c\sum_{i=0}^{ \lfloor (t-4)/2 \rfloor}F_{t-4-2i}(x)+F_{t-2}(x)=(c-1){G}_{t-4}(x)+{G}_{t-2}(x). 
\]
Since the zeros of $\mathscr{G}_{\epsilon,s-1}$ and $\mathscr{G}_{\epsilon,s}$ interlace on $(0,4(k-1))$, each zero of $(c-1)\mathscr{G}_{\epsilon,s-1}+\mathscr{G}_{\epsilon,s}$ is simple and belongs to $(0,4(k-1))$ except for the smallest zero. 
For $c=k$ the smallest zero is equal to $0$ because 
$
(k-1)\mathscr{G}_{\epsilon,s-1}(0)+\mathscr{G}_{\epsilon,s}(0)=0
$ 
by \eqref{eq:new_3-term} and \eqref{eq:def_scrG}. 
For $c>k$, the smallest zero is negative. From $x^{\epsilon}(x^2-k^2)((c-1)\mathscr{G}_{\epsilon,s-1}(x^2)+\mathscr{G}_{\epsilon,s}(x^2))=(x^2-k^2)((c-1){G}_{2s-2+\epsilon}(x)+{G}_{2s+\epsilon}(x))$, 
each non-zero real eigenvalue of 
$\boldsymbol{B}(k,t,c)$ has multiplicity 1, 
and if $c>k$, then $\boldsymbol{B}(k,t,c)$ has 
imaginary eigenvalues.       

Let $\mathfrak{f}_1(x)$ be the polynomial  
\[
\mathfrak{f}_1(x)=\frac{((c-1){G}_{t-4}(x)+G_{t-2}(x))^2}{x^2-\theta^2}=\sum_{i=0}^{ t- 3} f_i \mathscr{F}_{0,i}(x^2).
\]
We show that $\mathfrak{f}_2(x)=\sum_{i=0}^{ t - 3} f_i \mathscr{F}_{0,i}(x)$ satisfies the condition of the linear programming bound from Theorem~\ref{thm:lp_bound} for bipartite graphs. 
Note that $\mathfrak{f}_2(k^2)=\mathfrak{f}_1(k)>0$, and $\mathfrak{f}_2(\lambda^2)=\mathfrak{f}_1(\lambda) \leq 0$ for each $\lambda \in [-\theta, \theta]$. 
It suffices to show that $f_i>0$ for each $i\in \{0,1,\ldots, t- 3\}$. 
 
The polynomial $\mathfrak{f}_1(x)$ can be expressed by 
\begin{align*}
\mathfrak{f}_1(x)&= 
\frac{(c-1){G}_{t-4}(x)+{G}_{t-2}(x)}{x^2-\theta^2}
\left(c\sum_{i=0}^{\lfloor t/2 \rfloor-2}{F}_{t-4-2i}(x)+{F}_{t-2}(x) \right)\\
&=x^{2\epsilon}
\frac{(c-1)\mathscr{G}_{\epsilon,\lfloor t/2 \rfloor-2 }(x^2)+\mathscr{G}_{\epsilon,\lfloor t/2 \rfloor-1 }(x^2)}{x^2-\theta^2}
\left(c\sum_{i=0}^{\lfloor t/2 \rfloor-2}\mathscr{F}_{\epsilon,i}(x^2)+
\mathscr{F}_{\epsilon,\lfloor t/2 \rfloor-1}(x^2)\right), 
\end{align*}
where $\epsilon=0$ if $t$ is even, 
and $\epsilon=1$ if $t$ is odd. Thus, 
\[
\mathfrak{f}_2(x)=x^{\epsilon}
\frac{(c-1)\mathscr{G}_{\epsilon,\lfloor t/2 \rfloor-2 }(x)+\mathscr{G}_{\epsilon,\lfloor t/2 \rfloor-1 }(x)}{x-\theta^2}
\left(c\sum_{i=0}^{\lfloor t/2 \rfloor-2}\mathscr{F}_{\epsilon,i}(x)+
\mathscr{F}_{\epsilon,\lfloor t/2 \rfloor-1}(x)\right).
\]

By Proposition~3.2 in \cite{CK07}, 
$g(x)=((c-1)\mathscr{G}_{\epsilon,\lfloor t/2 \rfloor-2}+\mathscr{G}_{\epsilon,\lfloor t/2 \rfloor-1})/(x-\theta^2)$ has  positive coefficients 
in terms of $\mathscr{G}_{\epsilon,0},\mathscr{G}_{\epsilon,1},\ldots, \mathscr{G}_{\epsilon,\lfloor t/2 \rfloor-2}$. This implies that $g(x)$ has positive 
coefficients in terms of $\mathscr{F}_{\epsilon,0},\mathscr{F}_{\epsilon,1},\ldots, \mathscr{F}_{\epsilon,\lfloor t/2 \rfloor-2}$.
Therefore $f_i>0$ for each $i=\{0,1,\ldots, t-3 \}$ by 
Lemma~\ref{lem:coe}. 

The polynomial $g(x)$ can be expressed by $g(x)=\sum_{i=0}^{\lfloor t/2 \rfloor-2}g_i \mathscr{F}_{\epsilon,i}(x)$.  
By Lemma \ref{lem:coe}, we have
\[
f_0=\sum_{i=0}^{\lfloor t/2 \rfloor-2} c k^{\epsilon} g_i \mathscr{F}_{\epsilon,i}(k^2)=c k^{\epsilon} g(k^2).
\]
By applying Theorem~\ref{thm:lp_bound} to the polynomial $\mathfrak{f}_2(x)$, we have
\begin{align*}
b(k,\theta) &\leq \frac{2\mathfrak{f}_2(k^2)}{f_0}=2k^\epsilon\left( \sum_{i=0}^{\lfloor t/2 \rfloor-2}\mathscr{F}_{\epsilon,i}(k^2)+\mathscr{F}_{\epsilon,\lfloor t/2 \rfloor-1}(k^2)/c \right) \\
&=2\left( \sum_{i=0}^{t-4} (k-1)^{i}+ \frac{(k-1)^{t-3}}{c}+
\frac{(k-1)^{t-2}}{c}
\right).
\end{align*}

By Theorem~\ref{thm:lp_bound}, the bipartite graph attaining the bound $M(k,t,c)$ has girth at least $2t-4$, and at most $t$ distinct eigenvalues.  Since the diameter is at most $t-1$, the graph satisfies $g\geq 2d-2$, where $g$ is the girth and $d$ is the diameter. By $g\geq 2d-2$, the graph becomes  a distance-regular graph \cite[Theorem 4.4]{ADF16}, \cite{TY00}, and it must have the quotient matrix $\boldsymbol{B}(k,t,c)$ for $1\leq c < k$, or $\boldsymbol{B}(k,t-1,1)$ for $c=k$ (see Proposition~\ref{prop:c2} below).  
Conversely the distance-regular graph with the quotient matrix $\boldsymbol{B}(k,t,c)$ clearly attains the bound $M(k,t,c)$. 
\end{proof}
Note that $\Gamma$ is a distance-regular graph with the quotient matrix  $\boldsymbol{B}(k,d+1,c)$ if and only if $\Gamma$ is a connected bipartite $k$-regular graph that has only $d+1$ distinct eigenvalues, and whose girth is at least $2d-2$.   
Table \ref{tab:1} shows the known examples attaining the bound $M(k,d+1,c)$ \cite[Section~6.11]{BCN}. 
\begin{table}
\caption{Known bipartite graphs meeting the bound $M(k,d+1,c)$}\label{tab:1}
\begin{tabular}{|c|c|c|c|c|c|} 
\hline
$k$& $\theta$ & $b(k,\theta)$& $d$& $c$ &Name \\
\hline 
$2$&$2\cos(2\pi/n)$ & $n$ (even) & $n/2$& 1 & $n$-cycle $C_n$  \\
$k$&$0$ & $2k$& $2$ & $1$ & Complete bipartite graph $K_{k,k}$   \\
$k$& $\sqrt{k-\tau}$ & $2 (1+k(k-1)/\tau )$ & 3& $\tau$ & {\small  Symmetric $(v,k,\tau)$-design }  \\ 
{\small $r^2-r+1$} &$r$ & {\small $2(r^2+1)\times$} &4 & 
{\small $(r-1)^2$}
& {\small  $pg(r^2-r+1,r^2-r+1,(r-1)^2)$}\\
& & {\small $(r^2-r+1)$} & &  &\\
$q$&$\sqrt{q}$& $2q^2$ &4&$q-1$&$AG(2,q)$ minus a parallel class   \\
$q+1$&$\sqrt{2q}$&$2\sum_{i=0}^3 q^i$ &4&1    & $GQ(q,q)$  \\
$q+1$&$\sqrt{3q}$ & $2\sum_{i=0}^5 q^i$ &6&1&$GH(q,q)$ \\
6 & 2& 162 & 4 & 2 &  $pg(6,6,2)$ \\
\hline
\end{tabular}
$AG(2,q)$: affine plane, $GQ(q,q)$: generalized quadrangle, $GH(q,q)$: generalized hexagon, \\
$pg$: partial geometry, $q$: prime power, $r$: power of 2, \\
We use the bipartite incidence graph of an incidence structure.  
\end{table}

\begin{example} \label{ex:1}
Recall that $v(k, \theta)$ denotes the maximum order of a connected (not necessarily bipartite) $k$-regular graph whose second largest eigenvalue is at most $\theta$. We have $v(3,1)=10$, which is attained by the Petersen graph \cite{CKNV16} and $b(3,1)=8$ from Table 1, which is attained by the bipartite incidence graph of the symmetric $(4,3,2)$-design. 
\end{example}

The following is the bipartite version of Theorem 5 in \cite{N}. 
\begin{corollary}
Let $\Gamma$ be a bipartite distance-regular graph of order $n$ with quotient matrix $\boldsymbol{B}(k,t,c)$ with respect to the distance-partition from a vertex. Then $\lambda_2(\Gamma)\leq \lambda_2(\Gamma')$ for any bipartite $k$-regular graph $\Gamma'$ of order $n$. 
\end{corollary}
\begin{proof}
Assume that there exists a graph $\Gamma'$ of order $n$ such that $\lambda_2(\Gamma')<\lambda_2(\Gamma)$. Then $\Gamma'$ also attains the bound from Theorem \ref{bound}. This implies that $\Gamma'$ must have the eigenvalue $\lambda_2(\Gamma)$, which is a contradiction.  
\end{proof}

Let $\mu^{(j)}$ ({\it resp}. $\lambda^{(j)}$) denote the largest zero of $F_j(x)$ ({\it resp}. $G_j(x)$). 

\begin{proposition}
For each $\theta\in [0,2\sqrt{k-1})$, there exist $t,c$ such that $\theta$ is the second largest eigenvalue of $\boldsymbol{B}(k,t,c)$.
\end{proposition}
\begin{proof}

Note that $\lambda^{(j)}<\mu^{(j)}$ for $j\geq 1$ because 
$G_j(x)=\sum_{i=0}^{\lfloor j/2 \rfloor}F_{j-2i}(x) > 0$ for  $x \geq \mu^{(j)}$.
The second eigenvalue $\lambda_2(t,c)$ of  $\boldsymbol{B}(k,t,c)$ is equal to the largest zero of 
$(c-1)G_{t-4}(x)+G_{t-2}(x)$. 
 Since the zeros of $\mathscr{G}_{\epsilon,\floor{t/2}-2}$ and $\mathscr{G}_{\epsilon, \floor{t/2}-1}$  interlace,  $\lambda_2(t,c)$ is a monotonically decreasing function in $c$. In particular, $\lim_{c\rightarrow \infty}\lambda_2(t,c)=\lambda^{(t-4)}$ with $t\geq 5$, $\lambda_2(t,1)=\lambda^{(t-2)}$, and 
$\lim_{c\rightarrow 0} \lambda_2(t,c)=\mu^{(t-2)}$. 
The largest zero $r^{(j)}$ of $G_j(x)+G_{j-1}(x)$ can be expressed by $r^{(j)}=2\sqrt{k-1} \cos \alpha$, where $\pi/(j+1) < \alpha < \pi/j$ \cite[Section III.3]{BIb}.  
For $\lambda^{(j)}=2\sqrt{k-1} \cos \beta$, 
it follows from $r^{(j)}<\lambda^{(j)}$ that $\beta < \alpha < \pi/j$. This implies that
the possible value $\lambda_2(t,c)$ is between $\lim_{c\rightarrow k}\lambda_2(4,c)=0$ and $\lim_{t\rightarrow \infty} \lambda_2(t,1)=2\sqrt{k-1}$.
Therefore for each $\theta \in \mathopen[0 ,2\sqrt{k-1}\mathclose)$, there exist $t$, $c$ such that $\lambda$ is 
the second eigenvalue of $\boldsymbol{B}(k,t,c)$. 
\end{proof}

Note that for $\theta \in ( \lambda^{(t-2)}, \mu^{(t-2)}]$, $\theta$ is the second eigenvalue of both $\boldsymbol{B}(k,t,c_1)$ and $\boldsymbol{B}(k,t+2,c_2)$ for some $c_1$, $c_2$ with $0\leq c_1 < 1$, $c_2 > 0$. By the following proposition, we may assume $c\geq 1$ in Theorem~\ref{bound} to obtain better bounds. 

\begin{proposition} \label{prop:c} 
Let $\theta\in \mathopen(\lambda^{(t-2)}, \mu^{(t-2)} \mathclose]$. 
Suppose $c_1$ and $c_2$ satisfy that  $0\leq  c_1 < 1$, $c_2 > 0$ and the second largest eigenvalues of $\boldsymbol{B}(k,t,c_1)$ and $\boldsymbol{B}(k,t+2,c_2)$ are $\theta$. 
Then we have $M(k,t,c_1)>M(k,t+2,c_2)$. 
\end{proposition}
\begin{proof}
Since $(c_1-1)G_{t-4}(\theta)+G_{t-2}(\theta)=0$ holds, we have
 \[
c_1=-\frac{G_{t-2}(\theta)-G_{t-4}(\theta)}{G_{t-4}(\theta)}=-\frac{F_{t-2}(\theta)}{G_{t-4}(\theta)}. 
\]
Similarly $c_2=-F_{t}(\theta)/G_{t-2}(\theta)$ holds. 
By $\theta > \lambda^{(t-2)}$, we have  $F_{t-2}(\theta)=-c_1 G_{t-4}(\theta)<0$ and $F_{t}(\theta)=-c_2G_{t-2}(\theta)<0$. 
It therefore follows that 
\begin{align*}
M(k,t,c_1)-M(k,t+2,c_2) 
& =2k(k-1)^{t-2}\big(\frac{1}{c_1}-1-\frac{1}{c_2}(k-1)^2\big)\\
&= 2k(k-1)^{t-2}\big(-\frac{G_{t-4}(\theta)}{F_{t-2}(\theta)}-1+(k-1)^2\frac{G_{t-2}(\theta)}{F_{t}(\theta)}\big)\\
&= 2k(k-1)^{t-2}\big(-\frac{G_{t-2}(\theta)}{F_{t-2}(\theta)}+(k-1)^2\frac{G_{t-2}(\theta)}{F_{t}(\theta)}\big)\\
&=\frac{2k(k-1)^{t-2}G_{t-2}(\theta)}{F_{t-2}(\theta)F_{t}(\theta)}\big(-F_{t}(\theta)+(k-1)^2F_{t-2}(\theta) \big)\\
&=\frac{2k(k-1)^{t-2}(k^2-\theta^2)G_{t-2}(\theta)^2}{F_{t-2}(\theta)F_{t}(\theta)}>0. \qquad \qedhere
\end{align*}
\end{proof}
For $\theta \in ( \lambda^{(t-4)}, \lambda^{(t-3)}]$, $\theta$ is the second eigenvalue of both $\boldsymbol{B}(k,t,c_1)$ and $\boldsymbol{B}(k,t-1,c_2)$ for some $c_1$, $c_2$ with $c_1 \geq 1$, $c_2 \geq 1$. It follows that 
\[
\frac{1}{c_1}=-\frac{G_{t-4}(\theta)}{F_{t-2}(\theta)}= -\frac{G_{t-2}(\theta)}{F_{t-2}(\theta)}+1
=-\frac{\lambda G_{t-3}(\theta)-(k-1)G_{t-4}(\theta)}{F_{t-2}(\theta)}+1
=-\frac{\lambda G_{t-3}(\theta)}{F_{t-2}(\theta)}-\frac{k-1}{c_1}+1,
\]
and hence
\begin{equation} 
\label{eq:k/c}
\frac{k}{c_1}=-\frac{\lambda G_{t-3}(\theta)}{F_{t-2}(\theta)}+1
\end{equation}
for $t\geq 4$. Thus, 
if $\theta=\lambda^{(t-3)}$, then $c_1=k$.    This implies that $k \leq  c_1$. 
By the following proposition, we may assume $1 \leq c < k$ in Theorem~\ref{bound} to obtain better bounds. 

\begin{proposition} \label{prop:c2}
Let $\theta\in \mathopen(\lambda^{(t-4)}, \lambda^{(t-3)} \mathclose]$. 
Suppose $c_1$ and $c_2$ satisfy that  $k\leq  c_1$, $c_2 \geq 1$ and the second largest eigenvalues of $\boldsymbol{B}(k,t,c_1)$ and $\boldsymbol{B}(k,t-1,c_2)$ are $\theta$.  
Then we have $M(k,t,c_1)\geq M(k,t-1,c_2)$. Moreover, equality holds if and only if 
$c_1=k$ and $c_2=1$.   
\end{proposition}
\begin{proof}
From 
$c_2=-F_{t-3}(\theta)/G_{t-5}(\theta)$ and $\eqref{eq:k/c}$, 
we have
\begin{align*}
\frac{M(k,t,c_1)- M(k,t-1,c_2)}{2(k-1)^{t-4}}&=
 1+\frac{k-1}{c_1}+\frac{(k-1)^2}{c_1}-\frac{1}{c_2}-\frac{k-1}{c_2}\\
&=1-\frac{1}{c_2}+(k-1)\left(
\frac{k}{c_1}-\frac{1}{c_2} \right)\\
&=1+\frac{G_{t-5}(\theta)}{F_{t-3}(\theta)}+(k-1)\left( 
-\frac{xG_{t-3}(\theta)}{F_{t-2}(\theta)}+1 +\frac{G_{t-5}(\theta)}{F_{t-3}(\theta)} \right)\\
&=\frac{G_{t-3}(\theta)}{F_{t-3}(\theta)}-(k-1)^2 \frac{F_{t-4}(\theta)G_{t-3}(\theta)}{F_{t-2}(\theta)F_{t-3}(\theta)}\\
&=\frac{G_{t-3}(\theta)}{F_{t-3}(\theta)F_{t-2}(\theta)}
(F_{t-2}(\theta)-(k-1)^2F_{t-4}(\theta))\\
&= (x^2-k^2)\frac{G_{t-3}(\theta)G_{t-4}(\theta)}{F_{t-3}(\theta)F_{t-2}(\theta)}\geq 0. 
\end{align*}
This implies the proposition. 
\end{proof}
The above results imply the following theorem. 
\begin{theorem} \label{thm:lambda} 
Let $\lambda^{(j)}$ be the largest zero of $G_j(x)$ for $j\geq 1$. Then $\bigcup_{j= 1 }^\infty (\lambda^{(j)},\lambda^{(j+1)}]= (0,2 \sqrt{k-1})$. If $t\geq 4$ satisfies $\lambda^{(t-3)} < \theta \leq \lambda^{(t-2)}$, then 
\[
b(k,\theta) \leq M(k,t,c)=2\left(\sum_{i=0}^{t-4} (k-1)^{i}+ \frac{(k-1)^{t-3}}{c}+\frac{(k-1)^{t-2}}{c}\right), 
\]  
where $c=-F_{t-2}(\theta)/G_{t-4}(\theta)$. 
\end{theorem}

The following results in \cite{HJ12, HJ14} are obtained as corollaries of Theorem~\ref{thm:lambda}. 
\begin{corollary}[{\cite{HJ12}}] \label{coro:0}
Let $\Gamma$ be a bipartite $n$-regular graph with $2m$ nodes.  
If 
$
\lambda_2(\Gamma) \leq \sqrt{n-1}
$,
then
\[
m \leq 1+\frac{n(n-1)}{n-\lambda_2^2(\Gamma)}  
\]
or, equivalently, 
\[
\lambda_2^2(\Gamma) \geq \sqrt{\frac{mn-n^2}{m-1}}. 
\]
\end{corollary}	
\begin{proof}
This is immediate by Theorem \ref{thm:lambda} for $t=4$. Indeed, 
since $\lambda^{(1)}$ is the largest zero of $G_1(x)=x$, we have
$\lambda^{(1)}=0$. 
Since $\lambda^{(2)}$ is the largest zero of $G_2(x)=x^2-(n-1)$, we have
$\lambda^{(2)}=\sqrt{n-1}$. 
Since $c=-F_{2}(\theta)/G_{0}(\theta)=n-\theta^2$, 
we have $M(n,4,c)/2=1+n(n-1)/(n-\theta^2)$. 
\end{proof}

\begin{corollary}[{\cite[Theorem 4]{HJ14}}] \label{coro:1}
Let $\Gamma$ be a bipartite $n$-regular graph with $2m$ nodes.  
If   
$
\sqrt{n-1} \leq \lambda_2(\Gamma)\leq \sqrt{2(n-1)}
$, 
then
\[
m \leq n+\frac{n(n-1)}{2n-\lambda_2^2(\Gamma)-1}. 
\]
\end{corollary}
\begin{proof}
This is immediate by Theorem \ref{thm:lambda} for $t=5$. Indeed, 
since $\lambda^{(2)}$ is the largest zero of $G_2(x)=x^2-(n-1)$, we have
$\lambda^{(2)}=\sqrt{n-1}$.
Since $\lambda^{(3)}$ is the largest zero of $G_3(x)=x(x^2-2(n-1))$, we have
$\lambda^{(3)}=\sqrt{2(n-1)}$. 
Since 
$
c=-F_{3}(\theta)/G_{1}(\theta)=2n-\theta^2-1
$,  
we have 
$
M(n,5,c)/2=n+n(n-1)/(2n-\theta^2-1)$.
\end{proof}

For $0<\theta \leq \sqrt{k-1}$, the inequality $b(k,\theta) \leq 2(\theta^4+\theta^2+1)$ was obtained by Teranishi and Yasuno \cite[Proposition 7.1]{TY00}. This bound is improved as follows. 
\begin{corollary} \label{coro:2}
If $k^{1/4} <\theta \leq \sqrt{k-1}$, then
\[
b(k,\theta) \leq 2\left(1+\frac{k-1}{k-\theta^2} + \frac{(k-1)^2}{k-\theta^2}\right) < 2(\theta^4+\theta^2+1). 
\]
\end{corollary}
\begin{proof}
Note that we have $\lambda^{(1)}=0$, $\lambda^{(2)}=\sqrt{k-1}$, and $c=-F_2(\theta)/G_0(\theta)=k-\theta^2$. By Theorem~\ref{thm:lambda}, for $0< \theta \leq \sqrt{k-1}$, we have
$b(k,\theta) \leq 2(1+(k-1)/(k-\theta^2) + (k-1)^2/(k-\theta^2))$. 
The inequality $1+\frac{k-1}{k-\theta^2} + \frac{(k-1)^2}{k-\theta^2} < \theta^4+\theta^2+1$ holds if and only if $k^{1/4}< \theta \leq \sqrt{k-1}$. The assertion therefore follows. 
\end{proof}

By Theorem~\ref{bound}, the following is immediate.
\begin{corollary} \label{coro:Alon}
Let $\Gamma$ be a connected bipartite $k$-regular graph of order $v$. If $\theta$ is the second largest eigenvalue of $\boldsymbol{B}(k,t,c)$ and $v\geq M(k,t,c)$, then $\lambda_2(\Gamma) \geq \theta$ holds.
\end{corollary}
Li and Sol\'{e} \cite[Theorems 3 and 5]{CLS96} showed that if $\Gamma$
is of girth $g=2l$, then $\lambda_2(\Gamma) \geq 2\cos (\pi/l)$.
Corollary~\ref{coro:Alon} improves this result because we have $v \geq
M(k,l+1,1)$ when $g=2l$ and $\theta=2\cos (\pi/l)$ for $\boldsymbol{B}(k,l+1,1)$.

We prove that the bound \eqref{bktau} is better than the bound $\eqref{vktau}$ for any $k$ and $\theta$. For $\eqref{vktau}$ we have a similar theorem to Theorem~\ref{thm:lambda}. For $j\geq 0$, denote $\mathcal{G}_j(x)=\sum_{i=0}^j F_i(x)$. 
\begin{theorem}[\cite{CKNV16}] \label{thm:v2}
Let $r^{(j)}$ be the largest zero of $\mathcal{G}_j(x)$ for $j\geq 1$. Then $\bigcup_{j= 1 }^\infty (r^{(j)},r^{(j+1)}]= (-1,2 \sqrt{k-1})$. If $t\geq 3$ satisfies $r^{(t-2)} < \theta \leq r^{(t-1)}$, then 
\[
v(k,\theta) \leq N(k,t,c)=1+\sum_{i=0}^{t-3}k (k-1)^{i}+\frac{k(k-1)^{t-2}}{c}, 
\]  
where $c=-F_{t-1}(\theta)/\mathcal{G}_{t-2}(\theta)$. 
\end{theorem}  

\begin{theorem}
Let $k \geq 2$ be an integer. For $\theta \in (0,2 \sqrt{k-1})$, let $M(k,t_1,c_1)$ and $N(k,t_2,c_2)$ be defined as in Theorems \ref{thm:lambda} and \ref{thm:v2}, where $c_1=-F_{t_1-2}(\theta)/G_{t_1-4}(\theta)$, $c_2=-F_{t_2-1}(\theta)/\mathcal{G}_{t_2-2}(\theta)$, $\lambda^{(t_1-3)}<\theta \leq \lambda^{(t_1-2)}$, and $r^{(t_2-2)}<\theta \leq r^{(t_2-1)}$. Then 
\begin{equation*}
M(k,t_1,c_1) \leq N(k,t_2,c_2).
\end{equation*} 
Equality holds only if $t_1=t_2=t+1, \theta=\lambda^{(t-1)}, c_1=1$, and $c_2=k$.  
\end{theorem}
\begin{proof}
Note that $\lambda^{(t-2)} < r^{(t-1)} < \lambda^{(t-1)}$ because  $\mathcal{G}_{t-1}(x)=G_{t-1}(x)+G_{t-2}(x)$ for any $t\geq 3$.

Because $\theta \in (0,2\sqrt{k-1})=\cup_{j\geq 3}(\lambda^{(j-2)},\lambda^{(j-1)}]$, there is $t\geq 3$ such that $\theta\in (\lambda^{(t-2)},r^{(t-1)}]\cup (r^{(t-1)},\lambda^{(t-1)}]$. We consider each of the two possible cases $\lambda^{(t-2)}<\theta \leq  r^{(t-1)}$ and $r^{(t-1)} < \theta \leq \lambda^{(t-1)} $ separately.

Suppose $\lambda^{(t-2)}<\theta \leq  r^{(t-1)}$. Then $t_1=t+1$ and $t_2=t$. From Theorem \ref{thm:v2}, a simple calculation yields that
\[
N(k,t_2,c_2)=2\sum_{i=0}^{t_2-3}(k-1)^i+(k-1)^{t_2-2}+\frac{k(k-1)^{t_2-2}}{c_2},
\]
and therefore, 
\begin{align*}
N(k,t,c_2)-M(k,t+1,c_1)&=\left(1+\frac{k}{c_2}-\frac{2k}{c_1} \right)(k-1)^{t-2}\\
&=\left(1-\frac{k\mathcal{G}_{t-2}(\theta)}{F_{t-1}(\theta)}+\frac{2kG_{t-3}(\theta)}{F_{t-1}(\theta)}\right)(k-1)^{t-2}\\
&=\left(1-\frac{kG_{t-2}(\theta)-kG_{t-3}(\theta)}{G_{t-1}(\theta)-G_{t-3}(\theta)}\right)(k-1)^{t-2}.
\end{align*}
Because the zeroes of $G_{t-2}$ and $G_{t-1}$ interlace, we get that $G_{t-1}(\theta)<0<G_{t-3}(\theta)$. Thus, $G_{t-1}(\theta)-G_{t-3}(\theta)<0$. If $G_{t-2}(\theta)-G_{t-3}(\theta)>0$, then it is clear that $N(k,t,c_2)>M(k,t+1,c_1)$. Otherwise, if 
$G_{t-2}(\theta)-G_{t-3}(\theta)<0$, then 
\begin{align*}
|kG_{t-2}(\theta)-kG_{t-3}(\theta)|-|G_{t-1}(\theta)-G_{t-3}(\theta)|
&=(k-1)G_{t-3}(\theta)+G_{t-1}(\theta)-kG_{t-2}(\theta)\\
&=(\theta - k)G_{t-2}(\theta)<0 
\end{align*}
which implies that $N(k,t_2,c_2)>M(k,t_1,c_1)$. 

Suppose $r^{(t-1)} <\theta \leq \lambda^{(t-1)}$. Then $t_1=t+1$ and $t_2=t+1$. Thus we have
\begin{align*}
N(k,t+1,c_2)&-M(k,t+1,c_1)=\left(k+1+\frac{k(k-1)}{c_2}-\frac{2k}{c_1}\right) (k-1)^{t-2}\\
&=\left(k+1-\frac{k(k-1)\mathcal{G}_{t-1}(\theta)}{F_t(\theta)}+\frac{2kG_{t-3}(\theta)}{F_{t-1}(\theta)}\right) (k-1)^{t-2}\\
&=\left(k+1-\frac{k(k-1)(G_{t-1}(\theta)+G_{t-2}(\theta))}{F_t(\theta)}+\frac{2k(G_{t-1}(\theta)-F_{t-1}(\theta))}{F_{t-1}(\theta)}\right) (k-1)^{t-2}\\
&=\left(-\frac{k(k-1)G_{t-1}(\theta)}{F_t(\theta)}+\frac{2kG_{t-1}(\theta)}{F_{t-1}(\theta)}-(k-1)\left(1+k \frac{G_{t-2}(\theta)}{F_t(\theta)}\right) \right)  (k-1)^{t-2}\\
&=\left(-\frac{k(k-1)G_{t-1}(\theta)}{F_t(\theta)}+\frac{2kG_{t-1}(\theta)}{F_{t-1}(\theta)}-\frac{(k-1)\theta G_{t-1}(\theta)}{F_t(\theta)} \right)  (k-1)^{t-2}\\
&=\left( \frac{kF_t(\theta)-k(k-1)F_{t-1} (\theta)}{F_t(\theta) F_{t-1}(\theta)}+\frac{kF_t(\theta)-\theta(k-1)F_{t-1} (\theta)}{F_t(\theta) F_{t-1}(\theta)}\right)  (k-1)^{t-2}G_{t-1}(\theta)\\
&=\left(\frac{k(\theta-k)\mathcal{G}_{t-1}(\theta)}{F_t(\theta) F_{t-1}(\theta)}+\frac{(\theta^2-k^2) G_{t-2}(\theta)}{F_t(\theta) F_{t-1}(\theta)} \right)  (k-1)^{t-2}G_{t-1}(\theta)\\
&=\left(\frac{k\mathcal{G}_{t-1}(\theta)+(\theta +k) G_{t-2}(\theta)}{F_t(\theta) F_{t-1}(\theta)} \right) (k-1)^{t-2}  (\theta-k) G_{t-1}(\theta)\geq 0. 
\end{align*}
Equality holds only if $G_{t-1}(\theta)=0$ meaning that $\theta=\lambda^{(t-1)}$. Since $c_1=\frac{-F_{t-1}(\theta)}{G_{t-3}(\theta)}=\frac{G_{t-3}(\theta)-G_{t-1}(\theta)}{G_{t-3}(\theta)}=1-\frac{G_{t-1}(\theta)}{G_{t-3}(\theta)}$ and $c_2=\frac{-F_{t}(\theta)}{\mathcal{G}_{t-1}(\theta)}=\frac{G_{t-2}(\theta)-G_{t}(\theta)}{G_{t-1}(\theta)+G_{t-2}(\theta)}=\frac{kG_{t-2}(\theta)}{G_{t-2}(\theta)}$, this means that
$c_1=1$, and $c_2=k$.  
\end{proof}

\section{Non-existence of certain distance-regular graphs}
\label{sec:bound_d}

In this section, we prove the non-existence of the graph that attains the bound in Theorem~\ref{bound} for $t>15$ and $t=12$.
Namely we prove the following.  
\begin{theorem} \label{thm:non_existence}
Let $k$ and $c$ be two integers such that $k\geq 3$ and $1\leq c\leq k-1$. If $d=11$ or $d\geq 15$, there is no distance-regular graph $\Gamma$ with the quotient matrix $\boldsymbol{B} (k,d+1,c)$. 
\end{theorem}
We prove Theorem~\ref{thm:non_existence} by the manner given by Fuglister \cite{F87}.
Let $x=(t+1/t)\sqrt{k-1}$. The polynomial $G_i(x)$ can be expressed by 
\[
G_i(x)=\frac{\sqrt{(k-1)^i}}{t^i(t^2-1)}(t^{2i+2}-1). 
\]
The characteristic polynomial of $\boldsymbol{B}(k,d+1,c)$ is
$(x^2-k^2)\big( (c-1)G_{d-3}(x)+G_{d-1}(x) \big)$, and we have the expression 
\begin{align*}
\mathscr{S}_{d}(x)&=(c-1)G_{d-3}(x)+G_{d-1}(x)\\
&=\frac{\sqrt{(k-1)^{d-1}}}{t^{d-1}(t^2-1)}\big( 
t^{2d}+\frac{c-1}{k-1}t^{2d-2}- \frac{c-1}{k-1}t^2-1 \big).
\end{align*}
Let $\theta$ be an eigenvalue of $\Gamma$ that is not $\pm k$.  Put $\theta=(\tau+1/\tau) \sqrt{k-1}$ for some complex number $\tau$. 
 Let $n$ be the order of $\Gamma$. The multiplicity $m_\theta$ is given by 
\[
m_\theta=\frac{nck(k-c)(k-1)^{d-2}}
{(k^2-\theta^2) \mathscr{S}_{d}'(\theta) f_{d-1}(\theta)},
\]
where $\mathscr{S}_{d}'(x)$ is the derivative with respect to $x$, $f_{d-1}=(x-1+c)G_{d-2}+(x-k+c)G_{d-3}-
(k-1)G_{d-4}$ and $G_{-1}=0$.  
From $\mathscr{S}_{d}(\theta)=0$, we can obtain 
\[
\tau^{2d-2}=\frac{(c-1) \tau^2+k-1}{(k-1)\tau^2+c-1}.
\]
Then we may calculate that
\[
f_{d-1}(\theta)= -\frac{c(k-c)\sqrt{(k-1)^{d-2}}}
{\tau^{d-2}((k-1)\tau^2+c-1)},
\]
and 
\[
\mathscr{S}_{d}'(\theta)=
\frac{2 \sqrt{(k-1)^{d-4}}[(d-1)(k-1)(c-1)(\tau^4+1)+
\big(d(k-1)^2+(d-2)(c-1)^2\big)\tau^2]}
{\tau^{d-2}(\tau^2-1)^2\big((k-1)\tau^2+c-1\big)}. 
\]
 From these equations, the multiplicity $m_\theta$ can be expressed by 
\begin{align}
m_\theta&= \frac{nk(k-1)(\tau^2-1)^2 
\big( (c-1)\tau^2 + k-1 \big) \big((k-1)\tau^2+c-1 \big)}
{2(\theta^2-k^2) \tau^2 [(d-1)(k-1)(c-1)(\tau^4+1)+
 \big(d(k-1)^2+(d-2)(c-1)^2 \big) \tau^2
]} \nonumber \\
&=\frac{nk\big(\theta^2-4(k-1)\big)\big((c-1)\theta^2+(k-c)^2 \big)}
{2(\theta^2-k^2)[(d-1)(c-1)\theta^2+d(k-c)^2+2(c-1)(k-c)]} \nonumber \\
&=\frac{nk(k-1)\big(\phi-4\big)\big((c-1)(k-1)\phi+(k-c)^2 \big)}
{2\big((k-1)\phi-k^2\big)[(d-1)(c-1)(k-1)\phi+d(k-c)^2+2(c-1)(k-c)]}, \label{eq:quad} 
\end{align}
where $\theta^2=(k-1)\phi$. Unless $(k,c)=(2,1)$, expression \eqref{eq:quad} gives a non-trivial rational quadratic polynomial in $\phi$. 

Set 
\[
H_d(x)=\frac{\mathscr{S}_d(x)}{x^\epsilon \sqrt{(k-1)^{d-3-\epsilon} }}, 
\]
where $\epsilon=1$ if $d$ is even, and $\epsilon=0$ if $d$ is odd. 
Let $z=x^2/(k-1)$. 
For $u=t^2$, we have $z=(t+1/t)^2=u+1/u+2$. We compute
\[
H_{2m+1-\epsilon}(z)=(c-1) P_{m-1,\epsilon}(z)+(k-1) P_{m,\epsilon}(z),
\]
where 
\[
P_{i,\epsilon}(z)=
\frac{ u^{2i+1-\epsilon}-1}
{u^{i-\epsilon}(u+1)^{\epsilon}(u-1)}. 
\] 
Note that $P_{i,\epsilon}(z)$ satisfy the recurrence relation
\[
P_{i,\epsilon}(z)=(z-2) P_{i-1,\epsilon}(z)-P_{i-2,\epsilon}(z)
\]
with the initial conditions $P_{0,\epsilon}(z)=1-\epsilon$, $P_{1,1}(z)=1$ and $P_{1,0}(z)=z-1$. This implies that 
$P_{i,\epsilon}(z)$ is a monic polynomial of degree $i$ 
with integer coefficients. 
Table~\ref{tb:prop_P} shows some useful identities 
involving polynomials $P_{i,\epsilon}(z)$. 
By \eqref{eq:quad}, the polynomial $H_{d}(z)$ must split over the rationals into factors of degree at most $2$.

\begin{table}[h] 
\begin{center}
\begin{tabular}{l} 
\hline
$d=2m+1-\epsilon$ \\ 
\hline
$P_{m,\epsilon}(z)= (u^d-1)/u^{m-\epsilon} (u-1) (u+1)^\epsilon $\\
$P_{m-1,\epsilon}(z)= (u^{d-2}-1)/u^{m-1-\epsilon} (u-1) (u+1)^\epsilon$\\
$P_{m-1,\epsilon}(z)+P_{m,\epsilon}(z)=(u+1)^{1-\epsilon}(u^{d-1}-1)/u^{m-\epsilon} (u-1) $ \\
$-P_{m-1,\epsilon}(z)+P_{m,\epsilon}(z)= (u^{d-1}+1)/u^{m-\epsilon} (u+1)^\epsilon$ \\
\hline
\end{tabular}
\caption{Identities involving $P_{i,\epsilon}(z)$} 
\label{tb:prop_P}
\end{center}
\end{table}

\subsection{Case analysis modulo 2}
Let $c'=c-1$ and $k'=k-1$. If $c'$ and $k'$ have a factor in common, we may still factor out the content of $H_d(z)$. Call the resulting polynomial $\hat{H}_d(z)$. For $\hat{H}_d(z)$ modulo 2, there are three cases A--C, which are listed in Table \ref{tb:mod2_1}. 
For natural numbers $n$ and $a$,  let ${\rm ord}_n(a)$ be the non-negative integer $s$ such that $a=n^s b$ and 
$b$ is an integer that is not divisible by $n$. Suppose ${\rm ord}_n(0)= \infty$.   

\begin{table}[h]
\begin{center}
\begin{tabular}{llll} \hline
Cases & Conditions& $\hat{H}_{d}(z) \pmod{2}$ & $d$ \\ \hline
A& ${\rm ord}_2(c')>{\rm ord}_2(k')$ & $P_{m,\epsilon}(z)$ &
$d=2^r w$ \\
B& ${\rm ord}_2(c')<{\rm ord}_2(k')$ & $P_{m-1,\epsilon}(z)$ &
$d-2=2^r w$ \\
C& ${\rm ord}_2(c')={\rm ord}_2(k')$ & $P_{m-1,\epsilon}(z)+P_{m,\epsilon}(z)$ &
$d-1=2^r w$ \\
\hline
\end{tabular}
 \caption{$\hat{H}_d(z)$ modulo 2, $w \in \{ 1,3,5\}$} \label{tb:mod2_1}
\end{center}
\end{table}  

Each root of $\hat{H}_d(z)$ is a root of one of the three irreducible polynomials of degree at most 2 over ${\rm GF}(2)$, which are listed in Table~\ref{tb:mod2_2}. There are also 
listed the results of the substitution $z=u+1/u+2$, as well as the multiplicative orders modulo 2 of the roots of the polynomials in $u$.

\begin{table}[h]
\begin{center}
\begin{tabular}{lll}
\hline 
$f(z)$ & $z=u+1/u+2$ & order of $u$ \\ \hline
$z$ & $(u+1)^2/u$ & 1\\
$z+1$ & $(u^2+u+1)/u$ & 3 \\
$z^2+z+1$ & $(u^4+u^3+u^2+u+1)/u^2$ & 5 \\ \hline
\end{tabular}
\caption{Irreducible polynomials over ${\rm GF}(2)$} \label{tb:mod2_2}
\end{center}
\end{table} 

If an expression $u^i-1$ occurs as a factor of $\hat{H}_d(z)$ modulo 2, then we must have $i=2^r w$ for $w\in \{1,3,5 \}$. 
From the identities in Table~\ref{tb:prop_P}, we can obtain 
the possible values for the diameter $d$ in Table~\ref{tb:mod2_1}. 

\subsection{Case analysis modulo 3}
 For $\hat{H}_d(z)$ modulo 3, there are three cases a--d, which are listed in Table \ref{tb:mod3_1}. If ${\rm ord}_3(c')={\rm ord}_3(k')=m$, then let $c''=c'/3^m$ and $k''=k'/3^m$.

\begin{table}[h]
\begin{center}
\begin{tabular}{llll} \hline
Cases & Conditions& $\hat{H}_{d}(z) \pmod{2}$ & $d$ \\ \hline
a& ${\rm ord}_3(c')>{\rm ord}_3(k')$ & $\pm P_{m,\epsilon}(z)$ &
$d=3^r w$ \\
b& ${\rm ord}_3(c')<{\rm ord}_3(k')$ & $\pm P_{m-1,\epsilon}(z)$ &
$d-2=3^r w$ \\
c& ${\rm ord}_3(c')={\rm ord}_3(k')$, $c''\equiv k'' \pmod{3} $ & $\pm (P_{m-1,\epsilon}(z)+P_{m,\epsilon}(z))$ &
$d-1=3^r w$ \\
d& ${\rm ord}_3(c')={\rm ord}_3(k')$, $c''\equiv - k'' \pmod{3} $ & $\pm (P_{m-1,\epsilon}(z)-P_{m,\epsilon}(z))$ &
$2d-2=3^r w$\\
\hline
\end{tabular}
 \caption{$\hat{H}_d(z)$ modulo 3, $w \in \{1,2,4,5,8,10\}$} \label{tb:mod3_1}
\end{center}
\end{table}

There are six irreducible polynomials of degree at most 2 over ${\rm GF}(3)$, which are listed in Table~\ref{tb:mod3_2}. 
We can obtain the possible values for the diameter $d$ in Table~\ref{tb:mod3_1} by a similar way to modulo 2. Here $w \in \{1,2,4,5,8,10 \}$.

\begin{table}[h]
\begin{center}
\begin{tabular}{lll}
\hline 
$f(z)$ & $z=u+1/u+2$ & order of $u$ \\ \hline
$z-1$ & $(u-1)^2/u$ & 1 \\
$z$ & $(u+1)^2/u$ & 2\\
$z+1$ & $(u^2+1)/u$ & 4 \\
$z^2-z-1$ & $(u^2-u-1)(u^2+u+1)/u^2$ & 8 \\
$z^2+1$ & $(u^4+u^3+u^2+u+1)/u^2$ & 5 \\ 
$z^2+z-1$ & $(u^4-u^3+u^2-u+1)/u^2$ & 10 \\ 
\hline
\end{tabular}
\caption{Irreducible polynomials over ${\rm GF}(3)$} \label{tb:mod3_2}
\end{center}
\end{table} 

\subsection{A bound for the diameter}
Using a method similar to the one of Fuglister \cite{F87}, we can 
obtain the possible values of $d$ in all cases A--C and a--d, which are listed in Table \ref{tb:d_list}. 

\begin{table}[h]
\begin{center}
\begin{tabular}{lll}
\hline
Case & & \\
\cline{1-2}
(mod 2) & (mod 3) & Possible values of $d$\\
\hline
A & a & 3--6,8,10,12,24 \\
&b & 3--6,8,10,12,20,32 \\
&c,d & 3--6,10,16\\ 
B & a & 3--6,8,10,12,18,162 \\
&b & 3--8,10,12,14,26 \\
&c,d & 3--7,10,82\\
C & a& 3--6,9,81\\
&b& 3--7,11,17\\
&c&3--7,9,11,13,25\\
&d&3--7,13  \\
\hline
\end{tabular}
\end{center}
\caption{Possible values of $d$} \label{tb:d_list}
\end{table}

We eliminate several choices of $d$ from Table~\ref{tb:d_list} in this subsection.  
\begin{proposition}\label{prop:1}
There does not exist a 
distance-regular graph $\Gamma$ with the quotient matrix $\boldsymbol{B} (k,d+1,c)$ for $d=17,18,20,32,81,82,162$. 
\end{proposition}
\begin{proof}
Using a computer, we can obtain the factorization of $c'P_{m-1,\epsilon}(z)+k'P_{m,\epsilon}(z)$ modulo $p$ into irreducible polynomials for given $d$,  $k'$ and $c'$.  

For $d=18,81,82,162$, we can find an irreducible polynomial of degree at least $3$ as a factor of $c'P_{m-1,\epsilon}(z)+k'P_{m,\epsilon}(z)$ over ${\rm GF}(5)$ for each pair $(c',k') \in {\rm GF}(5)\times {\rm GF}(5)$. 

For $d=20,32$, we can find an irreducible polynomial of degree at least $3$ as a factor of $c'P_{m-1,\epsilon}(z)+k'P_{m,\epsilon}(z)$ over ${\rm GF}(7)$ for each pair $(c',k')\in {\rm GF}(7)\times {\rm GF}(7)$. 

For $d=17$, we can find an irreducible polynomial of degree at least $3$ as a factor of $c'P_{m-1,\epsilon}(z)+k'P_{m,\epsilon}(z)$ over ${\rm GF}(43)$ for each pair $(c',k')\in {\rm GF}(43)\times {\rm GF}(43)$.
\end{proof}

Bannai and Ito \cite{BI79} proved the unimodal property of the multiplicities of the eigenvalues of 
Moore polygons. After this work, they also proved the rationality of the eigenvalues of Moore polygons \cite{BI81}. 
The rationality of the eigenvalues is essential for the proof of the non-existence of Moore polygons \cite{DG81}.  
In our case, the unimodal property of the multiplicities for the positive eigenvalues is easy. 
\begin{lemma} \label{lem:unimodal}
Let $\Gamma$ be a distance-regular graph with the quotient matrix $\boldsymbol{B} (k,d+1,c)$. Let $d'=\lfloor (d-1)/2 \rfloor$, which is the number of the positive non-trivial eigenvalues. 
Let $\theta_1, \ldots, \theta_{d'}$ be the positive non-trivial eigenvalues of $\Gamma$ with $\theta_1>\cdots > \theta_{d'}$. 
Let $m_{\theta_i}$ be the multiplicity of $\theta_i$. 
Then it follows that 
\[
m_{\theta_1} < m_{\theta_2} < \cdots m_{\theta_{i-1}}< m_{\theta_i} \geq m_{\theta_{i+1}}>\cdots > m_{\theta_{d'}} \]
 for some $i \in \{1,\ldots, d'\}$.   
\end{lemma} 
\begin{proof}
The multiplicity $m_\theta$ of the eigenvalue $\theta$ can be expressed by equation \eqref{eq:quad}.  
The function $m_\theta$ has no pole for $0< \phi < 4$ and $k\geq 3$.   
This implies the unimodal property of the multiplicities. 
\end{proof}

It is known that 
\[
G_i(x)= (k-1)^{i/2} U_i (\frac{x}{2 \sqrt{k-1}}), 
\]
where $U_i$ is the Chebyshev polynomial of degree $i$, which is defined by 
$U_i(\cos \theta)=\sin((i+1)\theta)/ \sin \theta$ (see \cite{DSVb}). 
Thus the zeros of $G_i(x)$ are $2 \sqrt{k-1} \cos u_j^{(i)}$ for $j=1,\ldots, i$, where
$u_{j}^{(i)}=j \pi/(i+1)$. 
Since  the expression 
\[
\mathscr{S}_{d}(x)=(c-1)G_{d-3}(x)+G_{d-1}(x)
=x G_{d-2}(x)-(k-c) G_{d-3}(x), 
\]
the positive zeros $\theta_{i}=2\sqrt{k-1} \cos \alpha_i$ of $\mathscr{S}_{d}(x)$ with $0<\alpha_1<\cdots < \alpha_{d'}< \pi /2$ 
satisfy 
$
u^{(d-1)}_{i}<\alpha_i< 
u^{(d-2)}_{i}$ for each $i \in \{1,\ldots, d'\}$. 

Let $a(\phi)$ and $b(\phi)$ be the functions defined by
\begin{align*}
a(\phi)&= \frac{\phi -4}{(k-1) \phi - k^2}, \\
b(\phi)&= \frac{(c-1)(k-1)\phi+(k-c)^2}{(d-1)(c-1)(k-1)\phi+d(k-c)^2+2(c-1)(k-c)}\\
&=\frac{X_\phi}{(d-1)X_\phi+Y},
\end{align*}  
where $X_\phi=(c-1)(k-1)\phi+(k-c)^2$ and $Y=(k-c)(k+c-2)$. 
Note that $m_\theta=nk(k-1)a(\phi) b(\phi)/2$.  
Let $\theta_\alpha^2=(k-1) \alpha$ and 
$\theta_\beta^2=(k-1) \beta$. 
It follows that $m_{\theta_\beta}>m_{\theta_\alpha}$ if and only if 
\[
\frac{a(\beta)}{a(\alpha)}\cdot \frac{b(\beta)}{b(\alpha)}>1. 
\]
Let $\alpha=4 \cos^2 v$ and $\beta=4 \cos^2 w$. By direct calculation, 
\[
\frac{a(\beta)}{a(\alpha)}= 
1+ \frac{\cos 2v-\cos 2w}{1-\cos 2v} \cdot \frac{(k-2)^2}{(k-2)^2+2(k-1)(1-\cos 2w) } 
\]
and
\[
\frac{b(\beta)}{b(\alpha)}=1-\frac{Y(X_\alpha-X_\beta)}  
  {(d-1)X_\alpha X_\beta+Y X_\alpha}.
\]

Let 
\[
A=\frac{\cos 2v-\cos 2w}{1-\cos 2v} \cdot \frac{(k-2)^2}{(k-2)^2+2(k-1)(1-\cos 2w) } 
\]
and 
\[
B=\frac{Y(X_\alpha-X_\beta)}  
  {(d-1)X_\alpha X_\beta+Y X_\alpha}.
\]
Note that 
\[
\frac{a(\beta)}{a(\alpha)}\cdot \frac{b(\beta)}{b(\alpha)}=(1+A)(1-B)=1+B\big( A(\frac{1}{B}-1)-1\big). 
\]
If $A(1/B-1)>1$ holds, then $a(\beta)b(\beta)/(a(\alpha)b(\alpha))>1$. 

\begin{lemma} \label{lem:L}
If $\pi /4 < v < w < \pi /2$ holds, then it follows that 
\[
A(\frac{1}{B}-1) >L(v,w)
\]
where $L(v,w):=\frac{3\sqrt{3} (d-1)}{4}  \big(1-\frac{2}{k} \big)^2 
(1+\cos 2v) \sin^2 2 w$. 
\end{lemma}
\begin{proof}
We can calculate
\[
\frac{1}{B}-1=\frac{(d-1)X_\alpha X_\beta+Y X_\beta}{Y(X_\alpha-X_\beta)}. 
\]
Since it follows that 
\[X_\alpha
=\big((c-1) +(k-1) \cos 2v \big)^2+(k-1)^2(1-\cos^2 2v) \geq (k-1)^2\sin^2 2v,\] 
we have
\[
(d-1)X_\alpha X_\beta \geq  (d-1)(k-1)^4\sin^2 2v\sin^2 2w
\]
and \[
YX_\beta \geq (k-c)(k+c-2)(k-1)^2\sin^2 2w.
\] 
It follows that 
\begin{align*}
Y(X_\alpha-X_\beta)&=2(k-c)(k+c-2)(c-1)(k-1)(\cos 2v-\cos 2w).
\end{align*}

For $1\leq c \leq k-1$,  the function $f(c)=(k-c)(k+c-2)(c-1)$ is maximum at $c=1+(k-1)/\sqrt{3}$. It therefore follows that
\[
Y(X_\alpha-X_\beta)<\frac{4}{3\sqrt{3}}(k-1)^4(\cos 2v-\cos 2w). 
\]
Thus we obtain 
\begin{align*}
\frac{1}{B}-1&>
\frac{3\sqrt{3}\big( (d-1)(k-1)^2\sin^2 2v\sin^2 2w+(k-c)(k+c-2)\sin^2 2w\big)}{4(k-1)^2(\cos 2v-\cos 2w)}\\
&>\frac{3\sqrt{3} (d-1)\sin^2 2v\sin^2 2w}{4(\cos 2v-\cos 2w)},
\end{align*}
and hence
\begin{align*}
A(\frac{1}{B}-1) &>  \frac{\cos 2v-\cos 2w}{1-\cos 2v} \cdot \frac{(k-2)^2}{(k-2)^2+2(k-1)(1-\cos 2w) } \cdot
\frac{3\sqrt{3} (d-1)\sin^2 2v\sin^2 2w}{4(\cos 2v-\cos 2w)}\\
&>\frac{3\sqrt{3} (d-1)}{4} \cdot \frac{(k-2)^2}{(k-2)^2+4(k-1) }\cdot 
(1+\cos 2v) \sin^2 2 w\\
&=\frac{3\sqrt{3} (d-1)}{4}  \big(1-\frac{2}{k} \big)^2 
(1+\cos 2v) \sin^2 2 w  \qedhere
\end{align*}
\end{proof}


\begin{lemma} \label{lem:L>1}
Suppose $\alpha=4\cos^2 v=\theta_\alpha^2/(k-1)$, $\beta=4\cos^2 w=\theta_\beta^2/(k-1)$, and $\pi /4 < v < w < \pi /2$.  
If $L(v,w)\geq 1$, then $m_{\theta_\alpha}<m_{\theta_\beta}$. 
\end{lemma}
\begin{proof}
By Lemma~\ref{lem:L} and $B>0$, we have
\[
\frac{m_{\theta_\beta}}{m_{\theta_\alpha}}=\frac{a(\beta)}{a(\alpha)}\cdot \frac{b(\beta)}{b(\alpha)}=(1+A)(1-B)=1+B\big( A(\frac{1}{B}-1)-1\big)>1+B\big( L(v,w)-1\big)\geq 1.   \qedhere
\]
\end{proof}

\begin{lemma}\label{lem:num_irr}
Let $\Gamma$ be a distance-regular graph with the quotient matrix $\boldsymbol{B} (k,d+1,c)$. 
Let $u_j^{(i)}=j \pi /(i+1)$ and $d'=\lfloor (d-1)/2 \rfloor$. 
If $L(u_{d'-j-1}^{(d-2)},u_{d'-j}^{(d-2)})\geq 1$ for some integer $j$ with $0 \leq j < d'-(d+3)/4$, 
then the number of positive eigenvalues $\theta$ of $\Gamma$ such that $\theta^2$ is irrational is less than or equal to $j$. 
\end{lemma}
\begin{proof}
The inequality $j < d'-(d+3)/4$ implies $u_{d'-j-1}^{(d-2)}>\pi/4$. 
If $\theta=2 \sqrt{k-1} \cos u$, we write $m_\theta=m(u)$. 
By Lemma \ref{lem:L>1}, we have
$
m(u_{d'-j-1}^{(d-2)})<m(u_{d'-j}^{(d-2)})
$.  
Note that the eigenvalues $\theta_i=2 \sqrt{k-1} \cos \alpha_i$ satisfy  
 $u^{(d-1)}_{i}<\alpha_i< u^{(d-2)}_{i}$. 
By $
m(u_{d'-j-1}^{(d-2)})<m(u_{d'-j}^{(d-2)})
$ and the unimodal property of the function $m(u)$, we have  $m(\alpha_{d'-j-1})<m(\alpha_{d'-j})$.  
Note that if $\theta^2$ is irrational, then $m_\theta=m_{\theta'}$ for some eigenvalue $\theta'$ with $\theta' \ne \theta$. 
The assertion therefore follows from $m(\alpha_{d'-j-1})<m(\alpha_{d'-j})$ and Lemma~\ref{lem:unimodal}.  
\end{proof}

\begin{proposition} \label{prop:2}
There does not exist a 
distance-regular graph $\Gamma$ with the quotient matrix $\boldsymbol{B} (k,d+1,c)$ for  $d=11, 16,24,25,26$. 
\end{proposition}
\begin{proof}
By Lemma~\ref{lem:num_irr}, we can estimate the number of irrational square eigenvalues $(k-1)\phi=\theta^2$ by checking $L(u_{d'-j-1}^{(d-2)},u_{d'-j}^{(d-2)})>1$ with a computer.  

For $d=11$, the number of  irrational $\phi$ is at most $1$ for $k\geq 5$. 
For $k=3,4$, we can find an irreducible polynomial of degree at least $3$ as a factor of $c'P_{m-1,\epsilon}(z)+k'P_{m,\epsilon}(z)$ over $\mathbb{Q}$ for each pair $(c',k')$ with $0\leq c'\leq k'-1$.  
We can find an irreducible polynomial of degree at least $3$ or two irreducible polynomials of degree $2$ as a factor of $c'P_{m-1,\epsilon}(z)+k'P_{m,\epsilon}(z)$ over ${\rm GF}(2)$  for each pair $(c',k') \in {\rm GF}(2)\times {\rm GF}(2)$. 

For $d=16$, the number of  irrational $\phi$ is at most $2$ for $k\geq 3$. 
We can find an irreducible polynomial of degree at least $3$ or three irreducible polynomials of degree $2$ as a factor of $c'P_{m-1,\epsilon}(z)+k'P_{m,\epsilon}(z)$ over ${\rm GF}(3)$  for each pair $(c',k') \in {\rm GF}(3)\times {\rm GF}(3)$. 

For $d=25$, the number of  irrational $\phi$ is at most $2$ for $k\geq 6$.
For $k=3,4,5$, we can find an irreducible polynomial of degree at least $3$ as a factor of $c'P_{m-1,\epsilon}(z)+k'P_{m,\epsilon}(z)$ over $\mathbb{Q}$ for each pair $(c',k')$ with $0\leq c'\leq k'-1$.  
We can find an irreducible polynomial of degree at least $3$ or three irreducible polynomials of degree $2$ as a factor of $c'P_{m-1,\epsilon}(z)+k'P_{m,\epsilon}(z)$ over ${\rm GF}(3)$  for each pair $(c',k') \in {\rm GF}(3)\times {\rm GF}(3)$. 

For $d=24, 26$, the number of  irrational $\phi$ is at most $2$ for $k\geq 4$. 
For $k=3$, we can find an irreducible polynomial of degree at least $3$ as a factor of $c'P_{m-1,\epsilon}(z)+k'P_{m,\epsilon}(z)$ over $\mathbb{Q}$ for each $c'$ with $0\leq c'\leq k'-1$.
We can find an irreducible polynomial of degree at least $3$ or three irreducible polynomials of degree $2$ as a factor of $c'P_{m-1,\epsilon}(z)+k'P_{m,\epsilon}(z)$ over ${\rm GF}(3)$  for each pair $(c',k') \in {\rm GF}(3)\times {\rm GF}(3)$. 
\end{proof}

Theorem~\ref{thm:non_existence} follows
from Table~\ref{tb:d_list} and Propositions \ref{prop:1}, \ref{prop:2}.

\section{Conclusions}\label{sec:conclusions}

In this paper, we studied $b(k,\theta)$, the maximum number of vertices in bipartite regular graph of valency $k$ whose second largest eigenvalue is at most $\theta$. Our results extend previous work from \cite{CKNV16, HJ12, HJ14, CLS96,TY00}. Our general bound for $b(k,\theta)$ is attained whenever there exists a bipartite distance-regular graph of valency $k$, second largest eigenvalue $\theta$, girth $g$ and diameter $d$ with $g\geq 2d-2$. For $d=3$ and $g \geq 4$ all the point-block incidence graphs of symmetric designs give equality in our bound so the situation is well-understood. For $d\geq 4$ we only have the Van Lint-Schrijver geometry besides the generalized polygons. We believe that for $d \geq 5$ the only examples must have $c=1$ and are generalized polygons.

\bigskip

\noindent
\textbf{Acknowledgments.} 

The authors thank Tatsuro Ito for providing useful information relating to Section~\ref{sec:bound_d} and the two anonymous referees for their useful comments and suggestions.

\end{document}